\documentclass[12pt]{article}
\usepackage{amsmath,amssymb,amsthm,threeparttable}
\usepackage{graphicx}
\usepackage{overpic}
\usepackage{calc}

\normalsize\textwidth140mm\hoffset=-5mm%
\normalsize\parindent=1em%
\theoremstyle{plain} \newtheorem{theorem}{Theorem}
\theoremstyle{plain} \newtheorem{lemma}{Lemma}[section]
\theoremstyle{remark} \newtheorem{remark}[lemma]{Remark}
\theoremstyle{plain} \newtheorem{proposition}[lemma]{Proposition}
\theoremstyle{plain} \newtheorem{corollary}[lemma]{Corollary}
\theoremstyle{plain} \newtheorem{claim}[lemma]{Claim}
\theoremstyle{definition} \newtheorem{definition}[lemma]{Definition}
\theoremstyle{definition} 
\theoremstyle{definition} 
\newcommand{\adim}{\mathrm{D}}
\newcommand{\vdim}{d}
\newcommand{\edim}{\mathrm{D_{e}}}
\newcommand{\ch}{\mathrm{ch}}
\newcommand{\ha}{\hat{a}}
\newcommand{\eisp}[2]{M({#1},{#2})}

\newcommand{\onea}{\mathrm{1A}}
\newcommand{\honea}{\widehat{\onea}}
\newcommand{\twoa}{\mathrm{2A}}
\newcommand{\htwoa}{\widehat{\twoa}}
\newcommand{\threec}[1]{\mathrm{3C}({#1})}
\newcommand{\hthreec}[1]{\widehat{\mathrm{3C}}({#1})}
\newcommand{\fournp}[1]{\widehat{\mathrm{3C}^{00}}({#1})}
\newcommand{\bfournp}[1]{\mathrm{3C}^{00}(#1)}
\newcommand{\bfournpodd}[1]{\widehat{\mathrm{3C}^0}(#1)}
\newcommand{\bfournpmin}[1]{\mathrm{3C}^0({#1})}

\newcommand{\infalgnp}{\mathrm{Z}_{\textrm{np}}(\frac{1}{2})}

\newcommand{\chf}[1]{\ch\,\mathbb{#1}}

\newcommand{\1}{\mbox{1}\hspace{-0.25em}\mbox{l}}

\newcommand{\axdsum}[4]{\bigoplus_{#3=#1}^{#2}{#4}a_{#3}}
\newcommand{\haxdsum}[4]{\bigoplus_{#3=#1}^{#2}{#4}\ha_{#3}}

\newcommand{\ad}{\mathrm{ad}}
\newcommand{\frulej}[2]{\mathcal{F}_{#1}({#2})}

\newcommand{\chrg}[2]{\mathcal{X}_{#1}({#2})}

\newcommand{\catdecij}[3]{{#3}\text{-}\mathbf{Dec}_{{#1},{#2}}}
\newcommand{\cataxi}[2]{{#2}\text{-}\mathbf{Ax}_{#1}}
\newcommand{\cataxdeci}[3]{({#2},{#3})\text{-}\mathbf{AxDec}_{#1}}
\newcommand{\catdec}[1]{{#1}\text{-}\mathbf{Dec}}
\newcommand{\catax}[1]{{#1}\text{-}\mathbf{Ax}}
\newcommand{\cataxdec}[2]{({#1},{#2})\text{-}\mathbf{AxDec}}
\newcommand{\pr}{\mathrm{pr}}

\title{Universal decomposition algebras and the classification of 2-generated non-primitive axial algebras of Jordan type} 
\author{Takahiro Yabe\thanks{
Graduate School of Mathematical Sciences, The University of Tokyo
3-8-1, Komaba, Meguro-ku, Tokyo 153-8914, Japan, email: tyabe@ms.u-tokyo.ac.jp}}
\date{}

\begin{document}
\maketitle

\begin{abstract}\noindent
The classes of decomposition algebras, axial decomposition algebras and non-primitive axial algebras are shown to have universal algebras.
Furthermore, by using the existence of a universal algebra, 2-generated non-primitive axial algebras of Jordan type are classified.
\end{abstract}

\section{Introduction}
In \cite{hrs13}, J.\ I.\ Hall et al.\ introduced a class of commutative nonassociative algebras called axial algebras.
An axial algebra is generated by \textit{axes}, a distinguished set of idempotents, subject to a condition called a \textit{fusion rule} given in terms of the eigenvalues of the axes. 
For a fusion rule $\mathcal{F}$ and an algebra $A$, an idempotent $a$ of $A$ is called an {\it $\mathcal{F}$-axis}\/ if its action on $A$ is semisimple and the eigenspaces obey the fusion rule $\mathcal{F}$.
An algebra generated by $\mathcal{F}$-axes is called an {\it $\mathcal{F}$-axial algebra}.
There exist some results about the classification of axial algebras of various types.

In that paper, they constructed the universal $n$-generated Frobenius \textit{primitive} $\mathcal{F}$-axial algebra.
Here, an axis $a$ is said to be primitive if its 1-eigenspace is $\mathbb{F}a$ and an axial algebra is said to be primitive if it is generated by primitive axes and an algebra is said to be Frobenius if there exists a non-zero associative bilinear form.  
By using the universal algebra, they gave another proof of the classification given in \cite{i09} or \cite{s07} of 2-generated primitive axial algebras satisfying the fusion rule called Ising fusion rule as Table \ref{tabfrulei}.
Such algebras are called Norton-Sakuma algebras.
\begin{table}[h]
\begin{center}
\begin{tabular}{c|cccc}
$\star$&0&1&$\frac{1}{4}$&$\frac{1}{32}$\\ \hline
0&0&0&$\frac{1}{4}$&$\frac{1}{32}$\\ 
1&0&0&$\frac{1}{4}$&$\frac{1}{32}$\\ 
$\frac{1}{4}$&$\frac{1}{4}$&$\frac{1}{4}$&$\{0,1\}$&$\frac{1}{32}$\\ 
$\frac{1}{32}$&$\frac{1}{32}$&$\frac{1}{32}$&$\frac{1}{32}$&$\{0,1,\frac{1}{4}\}$
\end{tabular}
\caption{The Ising fusion rule}
\label{tabfrulei}
\end{center}
\end{table}

In \cite{hrs15}, primitive axial algebras of Jordan type $\eta$ were defined by J.\ I.\ Hall et al. and the classification with 2-generated case is completed.
Axial algebras of Jordan type obey the fusion rule as Table \ref{tabfrulej}. 
 \begin{table}[h]
\begin{center}
\begin{tabular}{c|ccc}
$\star$&0&1&$\eta$\\ \hline
0&0&$\emptyset$&$\eta$\\ 
1&$\emptyset$&1&$\eta$\\  
$\eta$&$\eta$&$\eta$&$\{0,1\}$
\end{tabular}
\caption{The Jordan type fusion rule}
\label{tabfrulej}
\end{center}
\end{table}

In \cite{r14}, F.\ Rehren deformed Norton-Sakuma algebras to axial algebras satisfying the fusion rule as Table \ref{tabfrulem}.
We say that an axial algebras is \textit{of Majorana type} $(\xi,\eta)$ if it satisfies such a fusion rule.
In \cite{y20}, we classified dihedral 2-generated primitive axial algebras of Majorana type $(\xi,\eta)$ in the case when $(\xi,\eta)\neq(\frac{1}{2},2)$ or the base field is not of characteristic 5.
We say that a 2-generated axial algebra is \textit{dihedral} if there exists an automorphism which switches the two generating axes.
\begin{table}[h]
\begin{center}
\begin{tabular}{c|cccc}
$\star$&0&1&$\xi$&$\eta$\\ \hline
0&0&0&$\xi$&$\eta$\\ 
1&0&0&$\xi$&$\eta$\\ 
$\xi$&$\xi$&$\xi$&$\{0,1\}$&$\eta$\\ 
$\eta$&$\eta$&$\eta$&$\eta$&$\{0,1,\xi\}$
\end{tabular}
\caption{The Majorana type fusion rule}
\label{tabfrulem}
\end{center}
\end{table}
For the case when $(\xi,\eta)=(\frac{1}{2},2)$ and the base field is of characteristic 5,  C.\ Franchi et al.\ determined such algebras in \cite{fms20}, \cite{fm21} and \cite{fms22}.
In these papers, the algebras we call dihedral 2-generated axial algebras of Majorana type are called 2-generated symmetric axial algebras of Monster type.
These results can be seen as a generalization of the result of \cite{s07}. 

In \cite{y21}, we classified dihedral 2-generated primitive axial algebra of Majorana type with $\xi=\eta$ by using the concept of axial decomposition algebras, which was considerd in \cite{dpsv19} by T.\ De Medts et al.

An $\mathcal{F}$-\textit{decomposition} of an algebra $M$ is a direct sum decomposition of $M$ satisfying the fusion rule $\mathcal{F}$.
For a fusion rule $\mathcal{F}=(S,\star)$ and a map $\lambda$ from $S$ to $\mathbb{F}$, an element $a$ of $M$ is said to be an $(\mathcal{F},\lambda)$-axis if there exists an $\mathcal{F}$-decomposition $(M_s)_{i\in S}$ such that $ax=\lambda_sx$ for all $x\in M_s$.
An $(\mathcal{F},\lambda)$-\textit{axial decomposition algebra} is an $\mathbb{F}$-algebra equippped with some $(\mathcal{F},\lambda)$-axes.

One of the purposes of this paper is to prove that there exists a universal object of the category of decomposition algebras with $n$ generators and $m$ decompositions for positive integers $n$ and $m$.
Here, an $\mathcal{F}$-\textit{decomposition algebra} is an $\mathbb{F}$-algebra equipped with some $\mathcal{F}$-decompositions and we say that an $\mathcal{F}$-decomposition algebra is generated by a subset $X$ as an $\mathcal{F}$-decomposition algebra if the algebra itself is the minimum subalgebra as an $\mathbb{F}$-algebra such that it is closed under the projections induced by the decompositions and includes $X$.

\begin{theorem}\sl\label{thunidecalg}
Let $\mathcal{F}=(S,\star)$ be a finite fusion rule and $n$ and $m$ positive integers.

Then the category of $n$-generated $\mathcal{F}$-decomposition algebras over $\mathbb{F}$ equipped with $m$ decompositions has a universal object.
\end{theorem}

This theorem means that there exists an $\mathcal{F}$-decomposition algebra $M$ such that any $n$-generated $\mathcal{F}$-decomposition algebra equipped with $m$ $\mathcal{F}$-decompositions are isomorphic to a quotient of $M$.

By using the theorem, we can construct universal axial decomposition algebras and universal axial algebras.
\begin{corollary}\sl\label{thuniadecalg}
Let $\mathcal{F}=(S,\star)$ be a finite fusion rule, $\lambda$ a map from $S$ to $\mathbb{F}$ and $n$ a positive integer.

Then the category of $(\mathcal{F},\lambda)$-axial decomposition algebra over $\mathbb{F}$ generated by $n$ axes has a universal object.
\end{corollary}
\begin{corollary}\sl\label{corunialg}
Assume that $S$ is a finite subset of $\mathbb{F}$.

Then the category of $n$-generated $\mathcal{F}$-axial algebras has a universal object.
\end{corollary}
See Section \ref{sec3} for details.

As the other purpose, we classify 2-generated non-primitive axial algebras of Jordan type $\eta$.
By Corollary \ref{corunialg}, there exists a universal 2-generated non-primitive axial algebra of Jordan type $\eta$.
Thus it suffices to classify ideals of the universal algebra.

Let $\mathbb{F}$ be a field.
For $\Phi\subset\{0,1\}$  and $\eta\in\mathbb{F}\setminus\{0,1\}$, $\frulej{\Phi}{\eta}$ is a fusion rule as Table \ref{tabfruleieta}.
Note that in the primitive case, it is clear that $ax=0$ for each $0$-eigenvector $x$ of $a$ and thus $\Phi=\emptyset$.

\begin{table}[h]
\begin{center}
\begin{tabular}{c|ccc}
$\star$&0&1&$\eta$\\ \hline
0&0&$\Phi$&$\eta$\\ 
1&$\Phi$&1&$\eta$\\  
$\eta$&$\eta$&$\eta$&$\{0,1\}$
\end{tabular}
\caption{The fusion rule $\frulej{\Phi}{\eta}$}
\label{tabfruleieta}
\end{center}
\end{table}

\begin{theorem}\sl\label{thunialg}
Let $M_{\Phi}$ be the universal 2-generated $\frulej{\Phi}{\eta}$-axial algebra.
Then the following statements hold:
\begin{itemize}
\item[(1)]If $\Phi\neq\{0,1\}$ and $\eta\neq\frac{1}{2}$, then $M_{\Phi}$ is isomorphic to an algebra $\fournp{\eta}$.
\item[(2)]Assume that $\eta=\frac{1}{2}$.
If $\Phi=\emptyset$ or $\chf{F}\neq3$ nor 5, then $M_{\Phi}$ is isomorphic to an algebra $\infalgnp$.
\end{itemize}
\end{theorem}

\begin{theorem}\sl\label{propcalgj}
A 2-generated $\mathcal{F}_{\Phi}(\eta)$-axial algebra with $\Phi\neq\{0,1\}$ and $\eta\neq\frac{1}{2}$ nor $-1$ is isomorphic to one of the 10 algebras listed in Table \ref{tabfialg1} of Section \ref{sec6}.

A 2-generated $\mathcal{F}_{\Phi}(-1)$-axial algebra with $\Phi\neq\{0,1\}$ is isomorphic to one of the 16 algebras listed in Table \ref{tabfialg1} and Table \ref{tabfialg2} of Section \ref{sec6}.
\end{theorem}

See Section \ref{sec4} and \ref{sec6} for details.

Since an axial algebra is a pair of an algebra and its generating axes, there may exist two axial algebras such that they satisfy distinguished Fusion rules but they are isomorphic as an $\mathbb{F}$-algebra.
Then a question exists:
For an $\mathcal{F}$-axial algebra, does there exist an axis satisfying fusion rule other than $\mathcal{F}$?

To give an answer of that question, consider unital axial algebras.
In the 2-generated case, $\1-a$ is a primitive axis of Jordan type $1-\eta$ if $a$ is an axis of Jordan type $\eta$.
However, in $n$-generated cases with $n\geq3$, $\1-a$ is a non-primitive axis in general.
So we should consider non-primitive axes to generalize 2-generated cases to $n$-generated cases.

\vspace{\baselineskip}
In Section \ref{sec2}, we recall the definitions of axial algebras, their fundamental properties and relevant terminologies on general axial algebras.
In Section \ref{sec3}, we prove Theorem 1 and the corollaries.
In Section \ref{sec4'}, we show some properties of $\mathcal{F}_I$-axial algebras.
In Section \ref{sec4}, we list 2-generated non-primitive $\frulej{\Phi}{\eta}$-axial algebras.
In Section \ref{secasalg}, we consider the case which we call associative type, that is, the case when $\eta$-eigenspace is $0$. 
In Section \ref{sec5}, we determine the universal 2-generated $\frulej{\Phi}{\eta}$-axial algebra to show Theorem 2.
In Section \ref{sec6}, we prove Theorem 3.

\section{Preliminaries}\label{sec2}
Throughout the paper, let $\mathbb{F}$ be a field with $\chf{F}\neq2$ and $M$ a commutative nonassociative algebra over $\mathbb{F}$.

Let $\mathcal{F}=(S,\star)$ be a fusion rule.
We say that $\mathcal{F}$ is finite if $S$ is finite.

\subsection{Decomposition algebras}
Let $M$ be a commutative nonassociative algebra over $\mathbb{F}$.

An \textit{$\mathcal{F}$-decomposition} of $M$ is a direct sum decomposition $(M_s)_{s\in S}$ of $M$ such that $M_{s}M_{t}\subset\bigoplus_{i\in s\star t}M_i$ for all $s,t\in S$.

An \textit{$\mathcal{F}$-decomposition algebra} is a triple $(M,I,\Omega)$ where $M$ is a commutative nonassociative algebra over $\mathbb{F}$, $I$ is an index set and $\Omega$ is an $S\times I$-tuple $(M_s^i)_{s\in S,i\in I}$ of subspaces of $M$ such that $(M_s^i)_{s\in S}$ is an $\mathcal{F}$-decomposition of $M$ for all $i\in I$.
We say that $M$ is an $\mathcal{F}$-decomposition algebra if there is no use to say what $I$ and $\Omega$ are.

\begin{definition}
Let $(M,I,(M_s^i)_{s\in S,i\in I})$ be an $\mathcal{F}$-decomposition algebra.
We say that a subset $A$ of $M$ generates a subalgebra $N$ of $M$ as an $\mathcal{F}$-decomposition algebra if $N$ is the minimum subalgebra of $M$ containing $A$ such that the triple $(N,I,(N\cap M_i^s)_{i\in I,s\in s})$ is an $\mathcal{F}$-decomposition algebra.  
\end{definition}

$\catdec{\mathcal{F}}$ is the category having as objects $\mathcal{F}$-decomposition algebras.
For two objects $(M,I,(M_s^i)_{i\in I,s\in S})$ and $(N,J,(N_s^j)_{j\in J,s\in S})$, a morphism from $(M,I,(M_s^i)_{i\in I,s\in S})$ to $(N,J,(N_s^j)_{j\in J,s\in S})$ is a pair $(\psi,\varphi)$ of a map $\psi$ from $I$ to $J$ and an $\mathbb{F}$-algebra  homomorphism $\varphi$ from $M$ to $N$ such that 
$$\varphi(M_s^i)\subset N_s^{\psi(i)}$$
for all $i\in I$ and $s\in S$.

\subsection{Axial decomposition algebras and the Miyamoto automorphisms}
Let $\lambda$ be a map from $S$ to $\mathbb{F}$ which sends $s$ to $\lambda_s$, $a$ an element of $M$ and $\Phi=(M_s)_{s\in S}$ an $\mathcal{F}$-decomposition of $M$.
Then the pair $(a,\Phi)$ is called an \textit{$(\mathcal{F},\lambda)$-axis} if 
$$ax=\lambda_sx$$
for all $s\in S$ and $x\in M_s$.
We say that $a$ is an axis if there is no use to say what $\Phi$ is.

An $(\mathcal{F},\lambda)$-axial decomposition algebra is a quadraple $(M,I,\Omega,\mathcal{A})$ of a commutative nonassociative algebra $M$, an index set $I$, an $S\times I$-tuple $\Omega=(M_s^i)_{s\in S,i\in I}$ of subspaces of $M$ and an $I$-tuple $\mathcal{A}=(a_i)_{i\in I}$ of elements of $M$ such that $(a_i,(M_s^i)_{s\in S})$ is an $(\mathcal{F},\lambda)$-axis for all $i\in I$.

$\cataxdec{\mathcal{F}}{\lambda}$ is the category having as objects $(\mathcal{F},\lambda)$-axial decomposition algebras.
For two objects $(M,I,(M_s^i)_{i\in I,s\in S},(a_i)_{i\in I})$ and $(N,J,(N_s^j)_{j\in J,s\in S},(b_j)_{j\in J})$, a morphism from $(M,I,(M_s^i)_{i\in I,s\in S},(a_i)_{i\in I})$ to $(N,J,(N_s^j)_{j\in J,s\in S},(b_j)_{j\in J})$ is a pair $(\psi,\varphi)$ of a map $\psi$ from $I$ to $J$ and an $\mathbb{F}$-algebra homomorphism $\varphi$ from $M$ to $N$ such that 
$$\varphi(M_s^i)\subset N_s^{\psi(i)}$$
and
$$\varphi(a_i)=b_{\psi(i)}$$
for all $i\in I$ and $s\in S$.

Let $G$ be a group and $\phi$ a map from $S$ to $G$ such that
$$\phi(s\star t)=\{\phi(s)\phi(t)\}$$
for all $s,t\in S$.
Put $g_s=\phi(s)$.

Let $\chrg{\mathbb{F}}{G}$ be the $\mathbb{F}$-character group of $G$.

For $\chi\in\chrg{\mathbb{F}}{G}$ and an $(\mathcal{F},\lambda)$-axes $(a,\Phi)$,
There exists an $\mathbb{F}$-algebra automorphism $\tau_{\chi,(a,\Phi)}$ such that
$$\tau_{\chi,(a,\Phi)}(x)=\chi(g_s)x$$
for all $s\in S$ and $x\in M_s$, where $\Phi=(M_s)_{s\in S}$.
We call such an automorphism an \textit{Miyamoto automorphism} of $\chi$ and $(a,\Phi)$.

\subsection{Axial algebras}
In this subsection, assume that $S\in\mathbb{F}$ and $\lambda_s=s$.

For an $(\mathcal{F},\lambda)$-axes $(a,\Phi)$, $\Phi=(\eisp{a}{s})_{s\in S}$.
Thus the axis is determined by $a$ only.
So we simply say that $a$ is an \textit{$\mathcal{F}$-axis}.

In this paper, we assume that an $\mathcal{F}$-axis is an idempotent.

An \textit{$\mathcal{F}$-axial algebra} is a pair $(M,\mathcal{A})$ of a commutative nonassociative algebra $M$ and a set $\mathcal{A}$ of $\mathcal{F}$-axes of $M$ such that $\mathcal{A}$ generates $M$ as an $\mathbb{F}$-algebra.
In the case when there is no danger of confusion, we simply say that $M$ is an $\mathcal{F}$-axial algebra.

For axial algebras, we can also define Miyamoto automorphisms by the same way as the case of axial decomposition algebras. 

$\catax{\mathcal{F}}$ is the category having as objects $\mathcal{F}$-axial algebras.
For two objects $(M,\mathcal{A})$ and $(N,\mathcal{B})$, a morphism from $(M,\mathcal{A})$ to $(N,\mathcal{B})$ is an $\mathbb{F}$-algebra homomorphism $\varphi$ from $M$ to $N$ such that
$$\varphi(\mathcal{A})\subset\mathcal{B}.$$

\section{Existence of universal algebras}\label{sec3}
The purpose of this section is to prove Theorem 1 , Corollaries \ref{thuniadecalg} and \ref{corunialg}.

Let $\mathcal{F}=(S,\star)$ be a finite fusion rule, $\lambda$ a map from $S$ to $\mathbb{F}$.
Set $\lambda_s=\lambda(s)$ for $s\in S$.
Let $n$ and $m$ be elemants of $\mathbb{Z}_{>0}$, $I=\{1,\ldots,n\}$ and $J=\{1,\ldots,m\}$.

\subsection{Category of $n$-generated decomposition algebras equipped with $m$ decompositions}
For two sets $A$ and $B$, $\mathcal{C}_{A,B}$ is the category having as objects triples $(M,(f_{b})_{b\in B},(x_a)_{a\in A})$ of a commutative nonassociative algebra $M$ over $\mathbb{F}$, a $B$-tuple $(f_{b})_{b\in B}$ of $\mathbb{F}$-space  endomorphisms of $M$ and a $A$-tuple $(x_a)_{a\in A}$ of elements of $M$ such that $M$ is the minimum subalgebra of $M$ which is invariant under $f_b$ for all $b\in B$ and includes $x_a$ for all $a\in A$.
A morphism from $(M,(f_b)_{b\in B},(x_a)_{a\in A})$ to $(N,(g_b)_{b\in B},(y_a)_{a\in A})$ is a $\mathbb{F}$-space homomorphism $\phi$ from $M$ to $N$ such that $\phi(x_a)=y_a$ for all $a\in A$ and $\phi\circ f_b=g_b\circ\phi$ for all $b\in B$.

The category $\catdecij{n}{m}{\mathcal{F}}$ is a full subcagtegory of $\mathcal{C}_{I,S\times J}$ having as its objects the triples $(M,(f_s^j)_{s\in S,j\in J},(a_i)_{i\in I})$ such that $(M,J,(f_s^j(M))_{s\in S,j\in J})$ is an object of $\catdec{\mathcal{F}}$.

We call the category $\catdecij{n}{m}{\mathcal{F}}$ the category of $\mathcal{F}$-decomposition algebras with $n$ generators and $m$  decompositions.

Let $(M,J,(M_s^j)_{s\in S,j\in J})$ be an $\mathcal{F}$-decomposition algebra generated by $(a_i)_{i\in I}$.
For $s\in S$ and $j\in J$, let $\pr_s^j$ be the projection from $M$ to $M_s^j$ induced by the direct sum decomposition $(M_s^j)_{s\in S}$. 
Then the triple $(M,(\pr_s^j)_{s\in S,j\in J},(a_i)_{i\in I})$ is an object of $\catdecij{n}{m}{\mathcal{F}}$. 

\subsection{Category of axial decomposition algebras generated by $n$ axes}
The category $\cataxdeci{n}{\mathcal{F}}{\lambda}$ is a full subcategory of $\mathcal{C}_{I,S\times I}$ having as objects triples $(M,(f_s^i)_{s\in S,i\in I}, (a_i)_{i\in I})$ such that $(M,I,(f_s^i(M))_{s\in S,i\in I}, (a_i)_{i\in I})$ is an $(\mathcal{F},\lambda)$-axial decomposition algebra.

We call the category $\cataxdeci{n}{\mathcal{F}}{\lambda}$ the category of $(\mathcal{F},\lambda)$-axial decomposition algebras with $n$ generating axes.

Let $(M,I,(M_s^i)_{s\in S,i\in I},(a_i)_{i\in I})$ be an $(\mathcal{F},\lambda)$-axial decomposition algebra such that $M$ is generated by $(a_i)_{i\in I}$ as an $\mathcal{F}$-decomposition algebra.
For $s\in S$ and $i\in I$, let $\pr_s^i$ be the projection from $M$ to $M_s^i$ induced by the direct sum decomposition $(M_s^i)_{s\in S}$. 
Then the triple $(M,(\pr_s^i)_{s\in S,i\in I},(a_i)_{i\in I})$ is an object of $\cataxdeci{n}{\mathcal{F}}{\lambda}$.

By the definition of axial decomposition algebras, $\cataxdeci{n}{\mathcal{F}}{\lambda}$ is a full subcategory of $\catdecij{n}{n}{\mathcal{F}}$.

\subsection{Category of $n$-generated axial algebras}
In this subsection, assume that $S$ is a finite subset of $\mathbb{F}$.

$\cataxi{n}{\mathcal{F}}$ is the category  having as objects pairs $(M,(a_i)_{i\in I})$ of a commutative nonassociative algebra $M$ over $\mathbb{F}$ and an $n$-tuple $(a_i)_{i\in I}$ of elements of $M$ such that $(M,\{a_i\}_{i\in I})$ is an $\mathcal{F}$-axial algebra.
A morphism from $(M,(a_i)_{i\in I})$ to $(N,(b_i)_{i\in I})$ is an $\mathbb{F}$-algebra homomorphism $\phi$ from $M$ to $N$ such that $\phi(a_i)=b_i$.

\begin{lemma}\sl\label{lemcatiso}
Assume that $S\subset\mathbb{F}$ and let $\lambda$ be the inclusion map from $S$ to $\mathbb{F}$.
Then the category $\cataxi{n}{\mathcal{F}}$ is isomorphic to a subcategory of $\cataxdeci{n}{\mathcal{F}}{\lambda}$ having as objects triples $(M,(\pr_s^i)_{s\in S,i\in I},(a_i)_{i\in I})$ such that $a_i$ is an idempotent for all $i\in I$.
\end{lemma}
\begin{proof}
It suffices to verify that if $M$ is generated by a set $\mathcal{A}$ of $\mathcal{F}$-axes as $\mathcal{F}$-decomposition algebra, then $M$ is generated by $\mathcal{A}$ as $\mathbb{F}$-algebra. 
Let
$$F_{A}(T)=\prod_{\alpha\in A}(T-\alpha)\in\mathbb{F}[T]$$
for $A\subset\mathbb{F}$ such that $|A|<\infty$.
Recall that $\pr_s^a$ is the projection from $M$ to $\eisp{a}{s}$ for $a\in\mathcal{A}$.
Then
$$\pr_s^a=(F_{S\setminus\{s\}}(s))^{-1}F_{S\setminus\{s\}}(\ad(a))$$
for  all $s\in S$ and $a\in\mathcal{A}$.
Thus the subalgebra of $M$ generated by $\mathcal{A}$ is invariant under $\pr_s^a$ for all $s\in S$ and $a\in\mathcal{A}$.  Hence it is generated by $\mathcal{A}$ as $\mathcal{F}$-decomposition algebra.
So $M$ is generated by $\mathcal{A}$ as $\mathbb{F}$-algebra.
\end{proof}

\subsection{Proof of Theorem 1}
\begin{lemma}\sl
Let $A$ and $B$ be sets.
Then the category $\mathcal{C}_{A,B}$ has a universal object.
\end{lemma}
\begin{proof}
For $k\in\mathbb{Z}_{>0}$, let $\hat{X}_{A,B}^k$ be a free nonassociative commutative magma satisfying the following statements:
\begin{itemize}
\item$\hat{X}_{A,B}^1$ is generated by $\{x_a\}_{a\in A}$.
\item For $k\in\mathbb{Z}_{>0}$, $\hat{X}_{A,B}^{k+1}$ is generated by $(B\times\hat{X}_{A.B}^{k})\cup\hat{X}_{A,B}^k$.
\end{itemize}
Let
$$\hat{X}_{A,B}=\bigcup_{k\in\mathbb{Z}_{>0}}\hat{X}_{A,B}^k.$$
Then $\mathbb{F}\hat{X}_{A,B}$ is an $\mathbb{F}$-algebra.

For $b\in B$ and $k\in\mathbb{Z}_{>0}$, set $\varphi_b^k$ be a map from $\hat{X}_{A,B}^k$ to $\hat{X}_{A,B}^{k+1}$ such that $\varphi_b^kx=(b,x)$.
Then the tuple $(\varphi_b^k)_{k\in\mathbb{Z}_{>0}}$ induces an endomorphism $\varphi_b$ of the $\mathbb{F}$-space $\mathbb{F}\hat{X}_{A.B}$.

Then it is easy to see that $(\mathbb{F}\hat{X}_{A.B},(\varphi_b)_{b\in B},(x_a)_{a\in A})$ is the universal object of $\mathcal{C}_{A,B}$.
\end{proof}

Let $(\hat{M}_{n,m},(\varphi_{(s,j)}^{n,m})_{s\in S,j\in J},(x_i^{I,J})_{i\in I})$ be the universal object of $\mathcal{C}_{I,S\times J}$.
For an object $(M,(f_s^j)_{s\in S,j\in J}, (a_i)_{i\in I})$ of $\mathcal{C}_{I,S\times J}$, there exists the unique morphism $\psi_M$ from $(\hat{M}_{n,m},(\varphi_{s,j}^{I,J})_{s\in S,j\in J},(x_i)_{i\in I})$ to $(M,(f_s^j)_{s\in S,j\in J}, (a_i)_{i\in I})$. 
Furthermore, $\psi_M$ is a surjective $\mathbb{F}$-algebra homomorphism from $\hat{M}_{n,m}$ to $M$.
Put $R_M$ as the kernel of $\psi_M$.
In such a case, we say that $R_M$ induces $(M,(f_s^j)_{s\in S,j\in J}, (a_i)_{i\in I})$.

\begin{lemma}\sl
An ideal $R$ of $\hat{M}_{n,m}$ induces an object of $\catdecij{n}{m}{\mathcal{F}}$ if and only if the following statements hold:
\begin{itemize}
\item[(1)]For all $j\in J$ and $x\in\hat{M}_{n,m}$,
$$x-\sum_{s\in S}\varphi_{s,j}^{n,m}x\in R.$$
\item[(2)]For all $j\in I$, $s,t,u\in S$ and $x\in\hat{M}_{n,m}$ such that $u\notin s\star t$, then 
$$\varphi_{u,j}^{I,J}((\varphi_{s,j}^{n,m}x)(\varphi_{t,j}^{n,m}y))\in R.$$
\item[(3)]For all $j\in J$, $s\in S$ and $x\in\hat{M}_{n,m}$,
$$((\varphi_{s,j}^{n,m})^2-\varphi_{s,j}^{n,m})x\in R.$$
\item[(4)]For all $j\in J$, $s,t\in S$ and $x\in\hat{M}_{n,m}$ such that $s\neq t$, 
$$\varphi_{s,j}^{n,m}\varphi_{t,j}^{n,m}x\in R.$$
\item[(5)]For all $j\in J$, $s\in S$ and $y\in J$, 
$$\varphi_{s,j}^{n,m}y\in R.$$
\end{itemize}
\end{lemma}
\begin{proof}
Assume that $J$ induces an object of $\catdecij{n}{m}{\mathcal{F}}$,
Then (5) holds since $\varphi_{s,j}^{I,J}$ induces an endomorphism $p_s^j$ of $\hat{M}_{n,m}/R$.
Since $p_s^i$ is a projection, (1), (3) and (4) hold.
By the fusion rule, (2) holds.

Assume that $R$ satisfies the five conditions.
Then  $\varphi_{s,j}^{n,m}$ induces an endomorphism $f_s^i$ of $\hat{M}_{n,m}/R$ since (5) holds and thus $R$ induces an object of $\mathcal{C}_{I,S\times J}$.
Furthermore, since (1), (2), (3) and (4) hold, the induced object is an object of $\catdecij{n}{m}{\mathcal{F}}$.
\end{proof}

\begin{proof}[Proof of Theorem 1]

For $k\in\mathbb{Z}_{>0}$, let $R^{n,m}_k$ be an ideal of $\hat{M}_{n,m}$ such that
\begin{itemize}
\item $R^{n,m}_1$ is generated by
\begin{eqnarray*}
&&\{x-\sum_{s\in S}\varphi_{s,j}^{n,m}x\mid j\in J,x\in\hat{M}_{n,m}\}\\
&&\cup\{((\varphi_{s,j}^{n,m})^2-\varphi_{s,j}^{n,m}x\mid j\in J,s\in S,x\in\hat{M}_{n,m}\}\\
&&\cup\{\varphi_{s,j}^{n,m}\varphi_{t,j}^{n,m}x\mid j\in J,s,t\in S,x\in\hat{M}_{n,m} \text{ such that }s\neq t\}\\
&&\cup\{\varphi_{u,j}^{n,m}((\varphi_{s,j}^{n,m}x)(\varphi_{t,j}^{n,m}y))\mid j\in J, x,y\in\hat{M}_{n,m}, s,t,u\in S\text{ such that }u\notin s\star t\},
\end{eqnarray*}
\item For $n\in\mathbb{Z}_{>0}$, $R^{n,m}_{k+1}$ is generated by 
$$\bigcup_{s\in S,i\in I}(\varphi_s^iR^{n,m}_k).$$
\end{itemize}
Let 
$$R^{n,m}=\sum_{k\in\mathbb{Z}_{>0}}R^{n,m}_k.$$

Then $R^{n,m}$ induces an object of $\catdecij{n}{m}{\mathcal{F}}$ and if $R$ induces an object of $\catdecij{n}{m}{\mathcal{F}}$, then $R^{n,m}\subset R$.
So the object of $\catdecij{n}{m}{\mathcal{F}}$ induced by $R^{n,m}$ is the universal object.
\end{proof}

\subsection{Proof of corollary \ref{thuniadecalg}}
Since $\cataxdeci{n}{\mathcal{F}}{\lambda}$ is a subcategory of $\mathcal{C}_{I,S\times I}$, for an object of $\cataxdeci{n}{\mathcal{F}}{\lambda}$, there exists an ideal $R$ of $\hat{M}_{n,n}$ which induces that object.
\begin{lemma}\sl
Assume that $R$ is an ideal of $\hat{M}_{n,n}$.
Then $R$ induces an object of $\cataxdeci{n}{\mathcal{F}}{\lambda}$ if and only if the following statements hold:
\begin{itemize}
\item[(1)]For $i\in I$, $s\in S$ and $x\in\hat{M}_{n,n}$,
$$(\ad(x_{i}^{n,n})-\lambda_s)\varphi_{s,i}^{n,n}x\in R,$$
\item[(2)]For $i\in I$, $s\in S$ and $x\in R$,
$$\varphi_{s,i}^{n,n}x\in R$$
\item[(3)]$R\supset R^{n,n}$
\end{itemize}
\end{lemma}
\begin{proof}
Assume that $R$ induces an object of $\cataxdeci{n}{\mathcal{F}}{\lambda}$.
Since $\cataxdeci{n}{\mathcal{F}}{\lambda}$ is a subcategory of $\catdecij{n}{n}{\mathcal{F}}$, (3) holds.
Since $\varphi$ induces an $\mathbb{F}$-space endomorphism of $\hat{M}_{n,n}/R$ for all $i\in I$, (2) holds.
By the definition of axial decomposition algebras, (1) holds.

Assume that $R$ satisfies the three conditions.
By (2) and (3), the object of $\mathcal{C}_{I,S\times I}$ induced by $R$ is an object of $\catdecij{n}{n}{\mathcal{F}}$.
Furthermore, by (1), that object is an object of $\cataxdeci{n}{\mathcal{F}}{\lambda}$ and thus the lemma is proved. 
\end{proof}

\begin{proof}[Proof of Corollary \ref{thuniadecalg}]
For $k\in\mathbb{Z}_{>0}$, let $\bar{R}^{n}_k$ be an ideal of $\hat{M}_{n,n}$ satisfies the following statements:
\begin{itemize}
\item$\bar{R}^{n}_1$ is generated by
$$\{(\ad(x_{i}^{n,n})-\lambda_s)x\mid i\in I,s\in S\text{ and }x\in\hat{M}_{n,n}\}.$$
\item For $n\in\mathbb{Z}_{>0}$, $\bar{R}^{n}_{k+1}$ is generated by
$$\bigcup_{s\in S,i\in I}(\varphi_{s,i}^{n,n}\bar{R}^{n}_k)$$
\end{itemize}
Let 
$$\bar{R}^n=\sum_{n\in\mathbb{Z}_{>0}}\bar{R}^n_k.$$

Then $R^{n,n}+\bar{R}^{n}$ induces an object of $\cataxdeci{n}{\mathcal{F}}{\lambda}$ and if an ideal $R$ of $\hat{M}_{n,n}$ induces an object of $\cataxdeci{n}{\mathcal{F}}{\lambda}$, then $R^{n,n}+\bar{R}^{n}\subset R$.
So the object of $\cataxdeci{n}{\mathcal{F}}{\lambda}$ induced by $R^{n,n}+\bar{R}^{n}$ is the universal object.
\end{proof}

\begin{remark}
The universal object of $\cataxdeci{n}{\mathcal{F}}{\lambda}$ is an initial object of a larger category of $(\mathcal{F},\lambda)$-axial decomposition algebras equipped with $n$ axes.
Furthermore, for an object of the larger category, the morphism from the initial object is unique.
Then the full subcategory of the larger category of objects such that the unique morphism from the initial object is surjective is isomorphic to $\cataxdeci{n}{\mathcal{F}}{\lambda}$.
\end{remark}

\subsection{Proof of Corollary \ref{corunialg}} 
Assume that $S\subset\mathbb{F}$ and $\lambda$ is the inclusion map from $S$ to $\mathbb{F}$.

By Lemma \ref{lemcatiso}, we may assume that $\cataxi{n}{\mathcal{F}}$ is a subcategory of $\cataxdeci{n}{\mathcal{F}}{\lambda}$ having as objects triples $(M,(\pr_s^i)_{s\in S,i\in n},(a_i)_{i\in n})$ such that $a_i$ is an idempotent for all $i\in n$.

Thus for an object of $\cataxi{n}{\mathcal{F}}$, there exists an ideal $R$ of $\hat{M}_{n,n}$ which induces that object of $\cataxi{n}{\mathcal{F}}$.

\begin{lemma}\sl\label{lemcataxalg}
Assume that $R$ is an ideal of $\hat{M}_{n,n}$.
Then $R$ induces an object of $\cataxi{n}{\mathcal{F}}$ if and only if the following statements hold:
\begin{itemize}
\item[(1)]For $i\in I$, $x_ix_i-x_i\in R$.
\item[(2)]$R\supset\bar{R}+R^{n,n}$.
\end{itemize}
\end{lemma}
\begin{proof}
Assume that $R$ induces an object of $\cataxi{n}{\mathcal{F}}$.
Since $\cataxi{n}{\mathcal{F}}$ is a subcategory of $\cataxdeci{n}{\mathcal{F}}{\lambda}$, (1) holds.
Since $\cataxi{n}{\mathcal{F}}$ is a subcategory of $\cataxdeci{n}{\mathcal{F}}{\lambda}$, (2) holds.

Assume that $R$ satisfies (1) and (2).
By (1),  $\cataxi{n}{\mathcal{F}}$ is a subcategory of $\cataxdeci{n}{\mathcal{F}}{\lambda}$.
So it suffices to verify that for $\varphi_{s,i}^{n,n}x\in R$ for all $i\in I$, $s\in S$ and $x\in R$.
Since $\lambda$ is an inclusion map from $S$ to $\mathbb{F}$ in this case,
$$\varphi_{s,i}^{n,n}x-(\prod_{t\in S}(s-t)^{-1}(\ad(x_i)-t))x\in\bar{R}+R^{n,n}\subset R$$
for all $s\in S$, $i\in I$ and $x\in\hat{M}_{n.n}$.
So $\varphi_{s,i}^{n,n}x\in R$ for all $i\in I$, $s\in S$ and $x\in R$.
Thus $R$ induces an object of $\cataxi{n}{\mathcal{F}}$.
\end{proof}
\begin{proof}[Proof of Corollary \ref{corunialg}]
It is easy to see that
$$(x_ix_i-x_i\mid i\in I)+\bar{R}+R_{n,n}$$
is the minimum ideal of $\hat{M}_{n,n}$ which induces an object of $\cataxi{n}{\mathcal{F}}$ by the lemma above.
Thus the object induced by such an ideal is the universal object.

\end{proof}

\begin{remark}
The arguments of this section hold if $I$ and $J$ are infinite.
Thus we can generalize the results of this section to the cases when the algebras have infinitely many axes or decompositions. 
\end{remark}

\section{Properties of non-primitive axial algebras of Jordan type}\label{sec4'}
From now on, let $\Phi$ be a subset of $\{0,1\}$ and $\eta\in\mathbb{F}\setminus\{0,1\}$.

Recall that $\frulej{\Phi}{\eta}$ is the fusion rule as Table \ref{tabfruleieta}.

For an $\frulej{\Phi}{\eta}$-axis $a$ of $M$, let $\tau_a$ denote the automorphism called the \textit{Miyamoto involution}\/ which is $1$ on $\eisp{a}{1}\oplus \eisp{a}{0}$ and $-1$ on $\eisp{a}{\eta}$ (See \cite{mi96}).
We say that a set $\mathcal{A}$ of $\frulej{\Phi}{\eta}$-axes of $M$ is \textit{Miyamoto closed} if for all $a,b\in\mathcal{A}$, $\tau_a(b)\in\mathcal{A}$.
For a set $\mathcal{A}$  of $\frulej{\Phi}{\eta}$-axes of $M$, the \textit{Miyamoto enclosure} of $\mathcal{A}$ is the minimum Miyamoto closed set including $\mathcal{A}$.

From now on, let $(M,\{a_0,a_1\})$ is an $\frulej{\Phi}{\eta}$-axial algebra.
Let $\tau_i=\tau_{a_i}$ for $i\in\{0,1\}$.

If there exists an automorphism of $M$ which sends $a_i$ to $a_{1-i}$, then we call such an automorphism the \textit{flip} and denote it by $\theta$.
The flip is unique if it exists.
If $M$ has the flip, then we say that $M$ is dihedral.
It is easy to see that the universal 2-generated $\frulej{\Phi}{\eta}$-axial algebra is dihedral. 

For each integers $i\in\mathbb{Z}$, let $a_{2i}=(\tau_1\tau_0)^i(a_0)$ and $a_{2i+1}=(\tau_1\tau_0)^i(a_1)$.
Then $a_i$ is an axis, $\theta(a_i)=a_{1-i}$, $\tau_0(a_i)=a_{-i}$ and $\tau_1(a_i)=a_{2-i}$ for all $i\in\mathbb{Z}$.
Let $\rho_i\in G$ be an automorphism such that $\rho_i=(\theta\tau_0)^i$ if it exists.
$\{a_i\mid i\in\mathbb{Z}\}$ is the Miyamoto enclosure of $\{a_0,a_1\}$.

Let $$p_{i,j}=a_ja_{i+j}-\eta(a_i+a_{i+j}),$$
$$x_i=p_{i,0}+\frac{\eta}{2}(a_i+a_{-i})$$
and
$$z_i=p_{i,0}+\eta a_0+\frac{\eta-1}{2}(a_i+a_{-i})$$
for $i,j\in\mathbb{Z}_{>0}$.
Then $x_i\in\eisp{a_0}{1}$, $z_i\in\eisp{a_0}{0}$ and
 $$a_0p_{i,0}=(1-\eta)p_{i,0}-\eta^2a_0+\frac{\eta-\eta^2}{2}(a_i+a_{-i})$$
since $a_i=x+z+\frac{1}{2}(a_i-a_{-i})$ for some $x\in\eisp{a_0}{1}$ and $z\in\eisp{a_0}{0}$.

It is easy to show that $p_{1,i}=p_{i,0}$ for all $i$.
From now on, let $p_1$ denote $p_{1,0}$.

\begin{lemma}\sl\label{lemprod1}
\begin{itemize}
\item[(1)]For all $\Phi\subset\{0,1\}$,
\begin{eqnarray*}
p_1p_1=&&\frac{2\eta-1}{2}a_0p_{2,1}-\frac{\eta^2}{2}p_{2,1}-(\eta^2-\eta)p_{2,0}+\frac{8\eta^2-12\eta+3}{2}p_1\\
&&+\frac{\eta(2\eta-1)^2}{2}a_0+\frac{3\eta(\eta-1)(2\eta-1)}{4}(a_1+a_{-1}).
\end{eqnarray*}
\item[(2)]If $\Phi\subset\{1\}$, then 
\begin{eqnarray*}
a_0p_{2,1}=&&p_{2,1}+(\eta-1)p_{2,0}-(2\eta-1)p_1\\
&&-\eta^2a_0-\frac{(\eta-1)(2\eta-1)}{2}(a_1+a_{-1})+\frac{(\eta-1)^2}{2}(a_2+a_{-2}).
\end{eqnarray*}
\item[(3)]If $\Phi\subset\{0\}$, then 
$$a_0p_{2,1}=\eta p_{2,0}-(2\eta-1)p_1-\eta^2a_0-\frac{\eta(2\eta-1)}{2}(a_1+a_{-1})+\frac{\eta^2}{2}(a_2+a_{-2})$$
\end{itemize}
\end{lemma}
\begin{proof}
Since $a_0(x_1x_1-z_1z_1)=x_1x_1$, (1) holds.

Since $a_0(x_1x_1-x_1z_1)=x_1x_1-x_1z_1$ if $\Phi\subset\{1\}$, (2) holds.

Since $a_0(x_1z_1-z_1z_1)=0$ if $\Phi\subset\{0\}$, (3) holds.
\end{proof}

\section{Examples of 2-generated non-primitive axial algebras of Jordan type}\label{sec4}
\subsection{Primitive case}
The 2-generated primitive axial algebras of Jordan type are classified in \cite{hrs15} by J.\ I.\ Hall et al.\
So see \cite{hrs15} for detail.

We recall algebras called  $\onea$, $\twoa$ and $\threec{\eta}$.

The algebra $\onea$ is an $\mathbb{F}$-space $\mathbb{F}a_0$ with $a_0a_0=a_0$.
This is formally a primitive axial algebra of Jordan type $\eta$ generated by $a_0$ and $a_0$.

The algebra $\twoa$ is an $\mathbb{F}$-space 
$$\twoa=\mathbb{F}a_0\oplus\mathbb{F}a_1$$
with multiplication given by
\begin{itemize}
\item$a_ia_i=a_i$ for $i\in\{0,1\}$,
\item$a_0a_1=0$.
\end{itemize}
This is a primitive axial algebra of Jordan type $\eta$ generated by $a_0$ and $a_1$.

Let $\eta\in\mathbb{F}\setminus\{0,1\}$.
$\threec{\eta}$ is an $\mathbb{F}$-space
$$\threec{\eta}=\axdsum{-1}{1}{i}{\mathbb{F}}$$
with multiplication given by
\begin{itemize}
\item$a_i\ha_i=a_i$
\item$a_i\ha_j=\frac{\eta}{2}(a_i+a_j-a_k)$ where $\{i,j,k\}=\{0,1,-1\}$.
\end{itemize}
This is a primitive axial algebra of Jordan type $\eta$ generated by $a_0$ and $a_1$.

There exists a quotient
$$\threec{-1}^{\times}=\threec{-1}/\mathbb{F}(a_0+a_1+a_{-1}).$$

\subsection{Associative case: Direct sums of $\onea$ and $\twoa$}
Let $(M,\mathcal{A}_1)$ and $(N,\mathcal{A}_2)$ be  non-primitive $\frulej{\Phi}{\eta}$-axial algebras.
Let $M\oplus N$ be an algebra with multiplication given by
$$(x_1+y_1)(x_2+y_2)=x_1x_2+y_1y_2$$
for all $x_1,x_2\in M$ and $y_1,y_2\in N$.

Then for $\frulej{\Phi}{\eta}$-axes $a\in\mathcal{A}_1$ and $b\in\mathcal{A}_2$, $a$, $b$ and $a+b$ are $\frulej{\Phi}{\eta}$-axes of $M\oplus N$.

So, we can construct a larger axial algebra by choosing a set of axes $$\mathcal{A}^{\prime}\subset\mathcal{A}_1\cup\mathcal{A}_2\cup\{a+b\mid a\in\mathcal{A}_1,b\in\mathcal{A}_2\}.$$

Consider 
$$\onea\oplus\onea=\mathbb{F}a\oplus\mathbb{F}b.$$
This algebra is associative.
There exist 3 $\frulej{\Phi}{\eta}$-axes $a$, $b$ and $a+b$.
Then the pairs  of axes $(a,b)$, $(a,a+b)$ and $(a+b,b)$ generate $\onea\oplus\onea$, respectively.

The pair $(\onea\oplus\onea,\{a,b\})$ is an $\frulej{\Phi}{\eta}$-axial algebra isomorphic to $\twoa$.

The pairs $(\onea\oplus\onea,\{a,a+b\})$ and $(\onea\oplus\onea,\{b,a+b\})$ are isomorphic $\frulej{\Phi}{\eta}$-axial algebras.
We call this algebra $\honea$.

Consider
$$\onea\oplus\twoa=\mathbb{F}a\oplus(\mathbb{F}s_0\oplus\mathbb{F}s_1).$$
This algebra is associative.
There exist axes $a$, $s_i$ for $i\in\{0,1\}$ and $a+s_i$ for $i\in\{0,1\}$.

Then the pair $(a+s_0,a+s_i)$ generates $\onea\oplus\twoa$.
Thus $(\onea\oplus\twoa,\{a+s_0,a+s_1\})$ is an $\frulej{\Phi}{\eta}$-axial algebra.
We call this algebra $\htwoa$.

Other pairs of axes do not generate $\onea\oplus\twoa$.

\subsection{Direct sums with $\threec{\eta}$}
Consider 
$$\threec{\eta}\oplus\onea=\axdsum{-1}{1}{i}{\mathbb{F}}\oplus\mathbb{F}s.$$
Then there exists a set of axes
$$\{a_i\mid i\in\{0,1\}\}\cup\{a_i+s\mid i\in\{0,1\}$$.

First, we consider the pair $(a_0+\hat{s},a_0+\hat{s})$.
This pair generates $\threec{\eta}\oplus\onea$ if $\eta\neq2$.
Thus $(\threec{\eta}\oplus\onea,\{a_0+s,a_1+s\})$ is an $\frulej{\Phi}{\eta}$-axial algebra if $\eta\neq2$.

We construct an $\frulej{\Phi}{\eta}$-axial algebra $\hthreec{\eta}$ which is isomorphic to $\threec{\eta}\oplus\onea$ if $\eta\neq2$ and is defined whether $\eta=2$ or not.

Let $\hthreec{\eta}$ be an $\mathbb{F}$-space
$$\hthreec{\eta}=\mathbb{F}\hat{q}\oplus\haxdsum{-1}{1}{i}{\mathbb{F}}$$
with multiplication given by
\begin{itemize}
\item$\hat{q}x=(\eta+1)x$ for all $x\in\hthreec{\eta}$.
\item$\hat{a}_i\hat{a}_i=\hat{a}_i$ for all $i\in\{0,1,-1\}$.
\item$\hat{a}_i\hat{a}_{i+1}=-\frac{1}{2}\hat{q}+\frac{\eta+1}{2}(\hat{a}_i+\hat{a}_{i+1})+\frac{1-\eta}{2}\hat{a}_{i-1}$ for all $i\in\{0,-1\}$ where $\hat{a}_2=\hat{a}_{-1}$.
\end{itemize}
Then $(\hthreec{\eta},\{\ha_0,\ha_1\})$ is an $\frulej{\Phi}{\eta}$-axial algebra and if $\eta\neq2$, then there exists an isomorphism from $\hthreec{\eta}$ to $\threec{\eta}\oplus\onea$ such that
$$\ha_i\mapsto a_i+s$$
for all $i\in\{0,1,-1\}$ and
$$\hat{q}\mapsto(\eta+1)s+a_0+a_1+a_{-1}.$$

There exists a quotient
$$\hthreec{-1}^{\times}=\hthreec{-1}/\mathbb{F}q.$$

Next, consider the pair $(a_0,\ha_1+\hat{s})$.
This pair generates $\threec{\eta}\oplus\onea$.
So the pair 
$$(\threec{\eta}\oplus\onea,(\ha_0,\ha_1+\hat{s}))$$
is also a 2-generated $\frulej{\Phi}{\eta}$-axial algebra.
We call this algebra $\bfournpmin{\eta}$.

For all $x\in\bfournpmin{\eta}$, 
$$x(a_0+a_1+a_{-1}+(\eta+1)s)=(\eta+1)x.$$
So $\bfournpmin{\eta}$ is unital if $\eta\neq-1$.
Furthermore, there exists an quotient as follows:
$$\bfournpmin{-1}^{\times}=\bfournpmin{-1}/\mathbb{F}(\ha_1+\ha_0+\ha_{-1}).$$

Consider 
$$\threec{\eta}\oplus\twoa=\axdsum{-1}{1}{i}{\mathbb{F}}\oplus(\mathbb{F}s_0\oplus\mathbb{F}s_1).$$
Then the pair 
$$(\threec{\eta}\oplus\twoa,(a_0+s_0,a_1+s_1))$$
 is a 2-generated $\frulej{\Phi}{\eta}$-axial algebra.
We call this algebra $\bfournp{\eta}$.

For all $x\in\bfournp{\eta}$, 
$$x(a_0+a_1+a_{-1}+(\eta+1)(s_0+s_1))=(\eta+1)x.$$
So $\bfournp{\eta}$ is unital if $\eta\neq-1$.
Furthermore, there exists a quotient 
$$\bfournp{-1}^{\times}=\bfournp{-1}/\mathbb{F}(\ha_0+\ha_1+\ha_{-1}).$$

\subsection{Direct sums with $\hthreec{\eta}$}
Consider $\hthreec{\eta}\oplus\onea$.
Let $s$ be the generating axis of $\onea$.
Then the pair 
$$(\hthreec{\eta}\oplus\onea,(\ha_0,\ha_1+s))$$
 is a 2-generated $\frulej{\Phi}{\eta}$-axial algebra.
We call this algebra $\bfournpodd{\eta}$.

Then for all $x\in\bfournpodd{\eta}$,
$$x(\hat{q}+(\eta+1)s)=(\eta+1)s.$$
Thus, if $\eta\neq-1$, then $\bfournpodd{\eta}$ is unital.
Furthermore, there exists a quotient
$$\bfournpodd{-1}^{\times}=\bfournp{-1}/\mathbb{F}\hat{q}.$$

Consider $\hthreec{\eta}\oplus\twoa$.
Let $(s_0,s_1)$ be the generating axes of $\twoa$.
Then the pair 
$$(\hthreec{\eta}\oplus\twoa,(\ha_0+s_0,\ha_1+s_1))$$
 is a 2-generated $\frulej{\Phi}{\eta}$-axial algebra.
We call this algebra $\fournp{\eta}$.

In fact, all of the other 2-generated $\frulej{\Phi}{\eta}$-axial algebras defined  before are quotients of the algebra $\fournp{\eta}$.

Let
$$r=\hat{q}-\ha_0-\ha_1-\ha_{-1}.$$
Then the following statements hold:
\begin{itemize}
\item$\bfournp{\eta}\cong\fournp{\eta}/\mathbb{F}r$.
\item$\bfournpodd{\eta}\cong\fournp{\eta}/\mathbb{F}s_0\cong\fournp{\eta}/\mathbb{F}s_1$.
\item$\bfournpmin{\eta}\cong\fournp{\eta}/\mathrm{Span}_{\mathbb{F}}\{s_0,r\}\cong\fournp{\eta}/\mathrm{Span}_{\mathbb{F}}\{s_1,r\}$.
\item$\hthreec{\eta}\cong\fournp{\eta}/\mathrm{Span}_{\mathbb{F}}\{s_0,s_1\}$.
\item$\threec{\eta}\cong\fournp{\eta}/\mathrm{Span}_{\mathbb{F}}\{s_0,s_1,r\}$.
\item$\htwoa\cong\fournp{\eta}/\mathrm{Span}_{\mathbb{F}}\{\ha_1-\ha_0,\ha_{-1}-\ha_0,\hat{q}-(\eta+1)\ha_0\}$.
\end{itemize}

For all $x\in\fournp{\eta}$,
$$x(\hat{q}+(\eta+1)(s_0+s_1))=(\eta+1)x.$$
Thus, $\fournp{\eta}$ is unital if $\eta\neq-1$.
Furthermore, there exists a quotient
$$\fournp{-1}^{\times}=\fournp{-1}/\mathbb{F}\hat{q}.$$

\subsection{Algebra $\infalgnp$ with $\eta=\frac{1}{2}$ and $\adim=\infty$}
Let $\infalgnp$ be an $\mathbb{F}$-space 
$$\infalgnp=\bigoplus_{i\in\mathbb{Z}_{>0}}\mathbb{F}\hat{p}_i\oplus\bigoplus_{i\in\mathbb{Z}}\mathbb{F}\ha_i$$
with multiplication given by
\begin{itemize}
\item$\ha_i\ha_j=\hat{p}_{|i-j|}+\eta(\ha_i+\ha_j)$,
\item$\ha_i\hat{p}_j=\frac{1}{2}\hat{p}_j-\frac{1}{4}\ha_i+\frac{1}{8}(\ha_{i-j}+\ha_{i+j})$,
\item$\hat{p}_i\hat{p}_j=-\frac{1}{4}(\hat{p}_i+\hat{p}_j)+\frac{1}{8}(\hat{p}_{i+j}+\hat{p}_{|i-j|})$.
\end{itemize}

Then $\infalgnp$ is an $\frulej{\Phi}{\frac{1}{2}}$-axial algebra for all $\Phi\subset\{0,1\}$.

\section{Universal 2-generated non-primitive axial algebra of associative type}\label{secasalg}
We will say that an $\mathcal{F}$-axial algebra is of \textit{associative type} if the fusion rule $\mathcal{F}$ is given as the following table.
\begin{table}[h]
\begin{center}
\begin{tabular}{c|cc}
$\star$&0&1\\ \hline
0&0&$\{0,1\}$\\ 
1&$\{0,1\}$&1
\end{tabular}
\caption{The associative type fusion rule}
\label{tabfrulea}
\end{center}
\end{table}

\begin{proposition}\sl
The universal 2-generated axial algebra of associative type is isomorphic to $\htwoa$.
\end{proposition}
\begin{proof}
Let $M$ be the universal 2-generated axial algebra of associative type and $a$ and $b$ the generating axes.
Put $x=a+b-ab$.

Then 
$$a+b-x\in\eisp{a}{1}\cap\eisp{b}{1},$$
$$x-a\in\eisp{a}{0}$$
and
$$x-b\in\eisp{b}{0}.$$
Thus $ax=a$ and $bx=b$.
Since
$$a((x-a-b)(x-a-b)-(x-a)(x-a))=(x-a-b)(x-a-b)$$
by the fusion rule, $xx=x$.

Thus there is a homomorphism from $\htwoa$ to $M$ which sends $a+s_0$ to $a$, $a+s_1$ to $b$ and $a+s_0+s_1$ to $x$.
Since $\htwoa$ is of associative type and $M$ is the universal algebra, $M$ is isomorphic to $\htwoa$.
\end{proof}

\begin{proposition}\label{qtwoa}\sl
A quotient of $\htwoa$ is isomorphic to $\onea$, $\honea$, $\twoa$ or $\htwoa$.
\end{proposition}
\begin{proof}
Let $R$ be an ideal of $\htwoa$.

Recall that for $x\in\htwoa$ and $i\in\{0,1\}$, $(a+s_i)x\in\eisp{a+s_i}{1}$ and $x-(a+s_i)x\in\eisp{a+s_i}{0}$.
Since $\htwoa$ is associative,
$$R=\bigoplus_{\iota,\kappa\in\{0,1\}}\eisp{a+s_0}{\iota}\cap\eisp{a+s_1}{\kappa}\cap R.$$
Since
$$\eisp{a+s_0}{1}\cap\eisp{a+s_1}{1}=\mathbb{F}a,$$
$$\eisp{a+s_0}{1}\cap\eisp{a+s_1}{0}=\mathbb{F}s_0,$$
$$\eisp{a+s_0}{0}\cap\eisp{a+s_1}{1}=\mathbb{F}s_1$$
and
$$\eisp{a+s_0}{0}\cap\eisp{a+s_1}{0}=0,$$
the proposition is proved.
\end{proof}

\section{Universal 2-generated non-primitive axial algebra of Jordan type}\label{sec5}
Recall that $\frulej{\Phi}{\eta}$ is the fusion rule given in Table \ref{tabfruleieta}, $M_\Phi(\eta)$ is the universal 2-generated $\frulej{\Phi}{\eta}$-axial algebra generated by $a_0$ and $a_1$ and then $M_\Phi(\eta)$ is dihedral.
Let $\theta$ be the flip, $\tau_i$ the Miyamoto involution of $a_i$, $\rho_n=(\theta\tau_0)^n$, $a_n=\rho_n(a_0)$ and $p_{i,j}=a_ja_{i+j}-\eta(a_i+a_{i+j})$.

Recall that
$$x_i=p_{i,0}+\frac{\eta}{2}(a_i+a_{-i})\in\eisp{a_0}{1}$$
and
$$z_i=p_{i,0}+\eta a_0+\frac{\eta-1}{2}(a_i+a_{-i})\in\eisp{a_0}{0}.$$

\subsection{The universal 2-generated $\frulej{\Phi}{\eta}$-axial algebra with $\eta\neq\frac{1}{2}$}
In this subsection, we prove Theorem 2 (1).

First, we determine the axial dimension of $M_{\Phi}(\eta)$.
\begin{lemma}
If $\Phi\neq\{0,1\}$ and $\eta\neq\frac{1}{2}$, then $a_2-a_{-2}+a_1-a_{-1}=0$.
\end{lemma}
\begin{proof}
First, we consider the case when $\Phi=\emptyset$.
By Lemma \ref{lemprod1} (2) and (3), 
\begin{eqnarray*}
p_{2,1}&&+(\eta-1)p_{2,0}-(2\eta-1)p_1\\
&&-\eta^2a_0-\frac{(\eta-1)(2\eta-1)}{2}(a_1+a_{-1})+\frac{(\eta-1)^2}{2}(a_2+a_{-2})\\
=&&\eta p_{2,0}-(2\eta-1)p_1-\eta^2a_0-\frac{\eta(2\eta-1)}{2}(a_1+a_{-1})+\frac{\eta^2}{2}(a_2+a_{-2}).
\end{eqnarray*}
Thus
$$p_{2,1}-p_{2,0}=\frac{2\eta-1}{2}(a_2+a_{-2}-a_1-a_{-1}).$$
Since $\theta(p_{2,1}-p_{2,0})=p_{2,0}-p_{2,1}$, $a_3+a_{-2}-a_1-a_0=0$.
Then $a_{-1}a_2=\rho_{-1}(a_0a_3)=p_1-p_{2,1}+\eta a_{-1}+(1-\eta)a_2$.
By the $\theta$-invariance, $p_{2,1}-p_{2,0}=(2\eta-1)(a_2-a_{-1})$.
Thus the lemma holds in this case.

Next, we assume that $\Phi=\{1\}$.
By Lemma \ref{lemprod1} (1)and (2),
\begin{eqnarray*}
p_1p_1=&&\frac{-(2\eta-1)^2}{2}p_{2,1}-\frac{\eta-1}{2}p_{2,0}+(2\eta^2-4\eta+1)p_1\\
&&+\frac{\eta(\eta-1)(2\eta-1)}{2}a_0+\frac{(\eta+1)(\eta-1)(2\eta-1)}{4}(a_1+a_{-1})\\
&&+\frac{(\eta-1)^2(2\eta-1)}{4}(a_2+a_{-2}).
\end{eqnarray*}
By the invariance,
$$(\eta-1)(a_3-a_{-3})+(\eta+1)(a_2-a_{-2}+a_{1}-a_{-1})=0.$$
Then
\begin{eqnarray*}
a_2a_{-2}=&&\rho_{-1}(a_{-1}a_3)\\
=&&-\frac{\eta+1}{\eta-1}p_{3,-1}-\frac{2}{\eta-1}p_{2,0}+\frac{\eta+1}{\eta}p_1+\eta a_{-2}+\frac{-\eta^2-2\eta+1}{\eta-1}a_2.
\end{eqnarray*}
By its $\tau_0$-invariance, $(\eta+1)(p_{3,1}-p_{3,-1})=(2\eta-1)(\eta+1)(a_2-a_{-2})$.
Thus 
$$(\eta+1)(a_3-a_{-3}+a_2-a_{-2}+a_1-a_{-1})=0$$
since $\rho_1(p_{3,1}-p_{3,-1})+\rho_{-1}(p_{3,1}-p_{3,-1})+p_{3,1}-p_{3,-1}=0$.
So $a_3=a_{-3}$ and $(\eta+1)(a_2-a_{-2}+a_1-a_{-1})=0$.
If $\chf{F}=3$, then $\eta+1\neq0$ since $\eta\neq\frac{1}{2}$.

Thus we may assume that $\chf{F}\neq3$.
By the $\theta$-invariance of $p_1p_1$,
\begin{eqnarray*}
a_3&=&a_{-2}-\frac{2}{\eta-1}(a_2-a_{-1})-a_1+a_0-\frac{2(\eta-2)}{(\eta-1)(2\eta-1)}(p_{2,1}-p_{2,0})\\
&=&\frac{\eta-3}{2(\eta-1)}(a_2+a_{-2}-a_1-a_{-1})+a_0-\frac{2(\eta-2)}{(\eta-1)(2\eta-1)}(p_{2,1}-p_{2,0}).
\end{eqnarray*}
Since $\rho_3(a_0a_3)=a_0a_3$,
\begin{eqnarray*}
0&=&a_3-a_0+\frac{\eta^2-3\eta+4}{\eta-1}(a_2+a_{-2}-a_1-a_{-1})+\frac{-2\eta^2+7\eta-9}{(\eta-1)(2\eta-1)}(p_{2,1}-p_{2,0})\\
&=&a_3-a_0-\frac{\eta-3}{2(\eta-1)}(a_2+a_{-2}-a_1-a_{-1})+\frac{2(\eta-2)}{(\eta-1)(2\eta-1)}(p_{2,1}-p_{2,0}).
\end{eqnarray*}
Thus
$$\frac{-2\eta^2+5\eta-5}{(\eta-1)(2\eta-1)}(p_{2,1}-p_{2,0})+\frac{2\eta^2-5\eta+5}{2(\eta-1)}(a_2+a_{-2}-a_{1}-a_{-1})=0.$$
Then it is shown that 
$$(2\eta^2-5\eta+5)(a_2-a_{-2}+a_1-a_{-1})=0$$
by the similar calculation as the case when $\Phi=\emptyset$.
Since $(\eta+1)(a_2-a_{-2}+a_1-a_{-1})=0$, $12(a_2-a_{-2}+a_{-1}-a_1)=0$.
So the lemma holds when $\chf{F}\neq3$.
Thus if $\Phi=\{1\}$, then the lemma holds.

Finally, we assume that $\Phi=\{0\}$.
Then by the similar calculation as the case when $\Phi=\{1\}$, we can show the properties bellow.
\begin{itemize}
\item$\eta(a_3-a_{-3})+(\eta-2)(a_2-a_{-2}+a_1-a_{-1})=0$ by the invariance of $p_1p_1$.
\item$a_3=a_{-3}$ and $(\eta-2)(a_2-a_{-2}+a_1-a_{-1})=0$ by the invariance of $a_2a_{-2}$.
\item$a_3=\frac{\eta+2}{2\eta}(a_2+a_{-2}-a_1-a_{-1})-a_0-\frac{2(\eta+1)}{\eta(2\eta-1)}(p_{2,1}-p_{2,0})$.
\item$0=-(p_{2,1}-p_{2,0})+a_3-a_0+\eta(a_2+a_{-2}-a_1-a_{-1})$ since $\rho_3(a_0a_3)=a_0a_3$.
\end{itemize}
Thus we can show that
$$\frac{2\eta^2+\eta+2}{2(\eta-1)(2\eta-1)}(2(p_{2,1}-p_{2,0})-(2\eta-1)(a_2+a_{-2}-a_1-a_{-1}))=0$$
and then 
$$(2\eta^2+\eta+2)(a_2-a_{-2}+a_{1}-a_{-1})=0.$$
Thus $12(a_2-a_{-2}+a_1-a_{-1})=0$.
So the lemma holds if $\chf{F}\neq3$.
If $\chf{F}=3$, then $\eta\neq2=\frac{1}{2}$ and thus the lemma holds since $(\eta-2)(a_2-a_{-2}+a_1-a_{-1})=0$.
\end{proof}

Then we can determine the multiplication of $M_\Phi(\eta)$.
Let
$$q=p_{2,0}-\frac{\eta+1}{\eta-1}p_1+(\eta-1)a_2-\eta a_{-1}-a_1-a_0.$$
Then $qa_0=\frac{\eta+1}{\eta-1}a_0$ and $qa_i=\frac{\eta+1}{\eta-1}a_0$ for all $i$ since $\rho_i(q)=q$ for all $i$.

Let 
$$r=p_{2,0}-\frac{\eta-2}{\eta}a_0+\eta a_2-(\eta-1)a_{-1}+a_1+a_0.$$
Then $ra_0=r$ and $ra_i=r$ for all $i$ since $\rho_i(r)=r$ for all $i$.

If $\Phi\subset\{1\}$, then $(a_0a_i)(q-\frac{\eta+1}{\eta-1}a_0)=a_0(a_i(q-\frac{\eta+1}{\eta-1}a_0))$ by the fusion rule.
So $qp_1=\frac{\eta+1}{\eta-1}p_1$ and $qq=\frac{\eta+1}{\eta-1}q$.
Thus $M_\Phi(\eta)$ is spanned by $\{a_i\mid-1\leq i\leq2\}\cup\{p_1,q\}$ and the dimension $\vdim\leq6$.
Since $\fournp{\eta}$ is 6-dimensional and is isomorphic to a quotient of $M_\Phi(\eta)$, $M_\Phi(\eta)$ is isomorphic to $\fournp{\eta}$.

If $\Phi\subset\{0\}$, then $(a_0a_i)r=a_0(a_ir)$ by the fusion rule.
So $p_{2,0}r=p_{2,1}r=p_1r=(1-2\eta)r$.
Then $M=\bigoplus_{i=-1}^{2}\mathbb{F}a_i\oplus\mathbb{F}p_1\oplus\mathbb{F}r$ and the multiplication is determined.
So $M_{\Phi}\cong M_{\emptyset}\cong\fournp{\eta}$.

Thus the claim is proved.

\subsection{The universal 2-generated $\mathcal{F}_\Phi(\frac{1}{2})$-axial algebra}
In this subsection we prove Theorem 2 (2).

In this subsection, assume that $\eta=\frac{1}{2}$.

\begin{claim}\sl\label{proprod2}
Assume that $p_{i,j}=p_{i,0}$ if $i\leq k$.
If $i,j\leq k$, then
$$p_{i,0}p_{j,0}=\frac{1}{4}p_{i+j,0}-\frac{1}{16}(p_{i+j,i}+p_{i+j,-i})-\frac{1}{4}p_{j,0}+\frac{1}{8}p_{|i-j|}-\frac{1}{4}p_{j,0}.$$
\end{claim}
\begin{proof}
Since $a_0(x_ix_j-z_iz_j)=x_ix_j$, the lemma holds.
\end{proof}

\begin{remark}
Assume that $\eta\neq\frac{1}{2}$.
Then as Lemma \ref{lemprod1} (1), there exists an undetermined term $a_0(p_{k+1,1}+p_{k+1,-1})$ in the right-hand side.
If $\Phi=\{0,1\}$, then the Lemma  \ref{lemprod1} (2) and (3) do not hold and the calculation becomes more difficult.
\end{remark}

Assume that $\Phi=\emptyset$ or $\chf{F}\neq3$ nor $5$.

\begin{claim}\sl
$p_{i,j}=p_{i,0}$ for all $i\in\mathbb{Z}_{>0}$ and $j\in\mathbb{Z}$.
\end{claim}
\begin{proof}
We prove the claim by induction.
Assume that if $i\leq k$, then $p_{i,j}=p_{i,0}$ for all $j\in\mathbb{Z}$.
Let $p_i$ denote $p_{i,0}$ for $i\leq k$.
Assume that $\chf{F}\neq3$ nor 5.

By Claim \ref{proprod2},
$$p_1p_k=\frac{1}{4}p_{k+1,0}-\frac{1}{16}(p_{k+1,1}+p_{k+1,-1})+\frac{1}{8}p_{k-1}-\frac{1}{4}(p_k+p_1).$$
If $k=1$, $p_{2,1}=p_{2,0}$ since $\theta(p_1p_1)=p_1p_1$ and $\chf{F}\neq3$.

Since $\rho_i(p_1p_k)=p_1p_k$, $p_{k+1,1}+p_{k+1.-1}-4p_{k+1,0}=p_{k+1,i}+p_{K+1,i-1}-4p_{k+1,i}$.
So if $\mathbb{F}[\sqrt{3}]\ni(2+\sqrt{3})^{k+1}\neq1$, then $p_{k+1,i}=p_{k+1,0}$ for all $i$.

If $(2+\sqrt{3})^{k+1}=1$, then
\begin{eqnarray*}
p_{k+1,i}=&&(2-\frac{(2+\sqrt{3})^i+(2-\sqrt{3})^i}{2})p_{k+1,0}\\
&&+\frac{(1+\sqrt{3})(2+\sqrt{3})^i-(1-\sqrt{3})(2-\sqrt{3})^i-2\sqrt{3}}{4\sqrt{3}}p_{k+1,1}\\
&&+\frac{(1+\sqrt{3})(2-\sqrt{3})^i-(1-\sqrt{3})(2+\sqrt{3})^i-2\sqrt{3}}{4\sqrt{3}}p_{k+1,-1}\\
\end{eqnarray*}
and $k\geq4$.
Since 
$$p_2p_{k-1}=\frac{1}{4}p_{k+1,0}-\frac{1}{16}(p_{k+1,2}+p_{k+1,-2})-\frac{1}{4}p_{k-1}+\frac{1}{8}p_{k-3}-\frac{1}{4}p_{2},$$
$0=p_{k+1,2}+p_{k+1.-2}-4p_{k+1,0}-(p_{k+1,3}+p_{k+1,-1}-4p_{k+1,1})=10(p_{k+1,0}-p_{k+1,1})$ and then $p_{k+1,i}=p_{k+1,0}$ for all $i$.

Next we assume that $\Phi=\emptyset$.
Since $x_1z_k=0$,
$$p_1p_k=\frac{1}{16}(p_{k+1,1}+p_{k+1,-1})+\frac{1}{8}p_{k-1}-\frac{1}{4}(p_k+p_1).$$

However, by Claim \ref{proprod2},
$$p_1p_k=\frac{1}{4}p_{k+1,0}-\frac{1}{16}(p_{k+1,1}+p_{k+1,-1})+\frac{1}{8}p_{k-1}-\frac{1}{4}(p_k+p_1).$$

So 
$$p_{k+1,0}=\frac{1}{2}(p_{k+1,1}+p_{k+1,-1})$$
and
$$p_1p_k=\frac{1}{8}p_{k+1,0}+\frac{1}{8}p_{k-1}-\frac{1}{4}p_k-\frac{1}{4}p_1.$$
Since $\rho_i(p_1p_k)=$, $p_{k+1.i}=p_{k+1,0}$ for all $i\in\mathbb{Z}$.

Thus the induction is completed.
\end{proof}

Let $p_i$ denote $p_{i,0}$.
Then by Claim \ref{proprod2}, 
$$p_ip_j=\frac{1}{8}(p_{i+j}+p_{|i-j|})-\frac{1}{4}(p_i+p_j)$$
for  all $i,j\in\mathbb{Z}_{>0}$.
So $M_\Phi(\frac{1}{2})$ is isomorphic to $\infalgnp$

\section{Proof of Theorem 3}\label{sec6}
In this section, we prove Theorem 3.
It suffices to verify the following lemma.
\begin{lemma}\sl\label{lemquo}
A Quotient of $\fournp{\eta}$ is isomorphic to one of the 10 algebras listed in Table \ref{tabfialg1} if $\eta\neq-1$.
If $\eta=-1$, then a quotient of $\fournp{\eta}$ is isomorphic to one of the 16 algebras listed in Table \ref{tabfialg1} and Table \ref{tabfialg2}. 
\end{lemma}

\begin{threeparttable}[h]
\begin{center}
\scalebox{0.8}{
{\def\arraystretch{1.6}
\begin{tabular}{c|c|c|c|c|c}
$M$&$|\mathcal{A}^c|$ \tnote{1}&$\adim$ \tnote{2}&$\vdim$ \tnote{3}&$\edim(a)$ \tnote{4}&$R_M$ \tnote{5}\\ \hline
$\onea$&1&1&1&$(1,0,0)$&$\bigoplus\limits_{i=-1}^1\mathbb{F}\ha_i\oplus\mathbb{F}(\hat{q}-(\eta+1)\ha_0)\oplus\mathbb{F}s_j$ with $j=0$ or $1$\\
&&&&&$\bigoplus\limits_{i=-1}^0\mathbb{F}(\ha_{i+1}-\ha_i)\oplus\mathbb{F}(\hat{q}-(\eta+1)\ha_0)\oplus\bigoplus\limits_{j=0}^1\mathbb{F}s_j$ \\ \hline
$\honea$&2&2&2&$(2,0,0)$&$\bigoplus\limits_{i=-1}^0\mathbb{F}(\ha_{i+1}-\ha_i)\oplus\mathbb{F}(\hat{q}-(\eta+1)\ha_0)\oplus\mathbb{F}s_j$ \\
&&&&or $(1,1,0)$&with $j=0$ or $1$\\ \hline
$\twoa$&2&2&2&$(1,1,0)$&$\bigoplus\limits_{i=-1}^1\mathbb{F}\ha_i\oplus\mathbb{F}(\hat{q}-(\eta+1)\ha_0)$\\ \hline
$\htwoa$&2&2&3&$(2,1,0)$&$\bigoplus\limits_{i=-1}^0\mathbb{F}(\ha_{i+1}-\ha_i)\oplus\mathbb{F}(\hat{q}-(\eta+1)\ha_0)$\\ \hline
$\threec{\eta}$&3&3&3&$(1,1,1)$&$\mathbb{F}s_0\oplus\mathbb{F}s_1\oplus\mathbb{F}(\hat{q}-(\ha_{-1}+\ha_1+\ha_0))$\\ \hline
$\hthreec{\eta}$&3&3&4&$(2,1,1)$&$\mathbb{F}s_0\oplus\mathbb{F}s_1$\\ \hline
$\bfournpmin{\eta}$&6&4&4&$(2,1,1)$&$\mathbb{F}s_i\oplus\mathbb{F}(\hat{q}-(\ha_{-1}+\ha_1+\ha_0))$ with $i=0$ or $1$ \\
&&&&or $(1,2,1)$&\\ \hline
$\bfournp{\eta}$&6&4&5&$(2,2,1)$&$\mathbb{F}(\hat{q}-(\ha_{-1}+\ha_1+\ha_0))$\\ \hline
$\bfournpodd{\eta}$&6&4&5&$(3,1,1)$&$\mathbb{F}s_i$ with $i=0$ or $1$\\
&&&&or $(2,2,1)$&\\ \hline
$\fournp{\eta}$&6&4&6&$(3,2,1)$&\{0\}
\end{tabular}
}}
\begin{tablenotes}
\item[1]$\mathcal{A}^c$ is the Miyamoto enclosure of the generating axes of $M$.
\item[2]$\adim=\dim_{\mathbb{F}}\mathrm{Span}_{\mathbb{F}}\mathcal{A}^c$.
\item[3]$\vdim=\dim_{\mathbb{F}}M$.
\item[4]$\edim(a)=(\dim_{\mathbb{F}}\eisp{a}{1},\dim_{\mathbb{F}}\eisp{a}{0},\dim_{\mathbb{F}}\eisp{a}{\eta})$ for $a\in\mathcal{A}^c$.
\item[5]$R_M$ is an ideal of $\fournp{\eta}$ such that $M\cong\fournp{\eta}/R_M$.
\end{tablenotes}
\caption{Quotients of $\fournp{\eta}$}
\label{tabfialg1}
\end{center}
\end{threeparttable}

\begin{threeparttable}[h]
\begin{center}
\scalebox{0.8}{
{\def\arraystretch{1.6}
\begin{tabular}{c|c|c|c|c|c}
$M$&$|\mathcal{A}^c|$&$\adim$&$\vdim$&$\edim(a)$&$R_M$\\ \hline
$\threec{-1}^{\times}$&3&2&2&$(1,0,1)$&$\mathbb{F}s_0\oplus\mathbb{F}s_1\oplus\mathbb{F}(\ha_{-1}+\ha_1+\ha_0)\oplus\mathbb{F}\hat{q}$\\ \hline
$\hthreec{-1}^{\times}$&3&3&3&$(2,0,1)$&$\mathbb{F}s_0\oplus\mathbb{F}s_1\oplus\mathbb{F}\hat{q}$\\ \hline
$\bfournpmin{-1}^{\times}$&6&3&3&$(2,0,1)$&$\mathbb{F}s_i\oplus\mathbb{F}(\ha_{-1}+\ha_1+\ha_0)\oplus\mathbb{F}\hat{q}$ with $i=0$ or $1$\\
&&&&or $(1,1,1)$&\\ \hline
$\bfournp{-1}^{\times}$&6&4&4&$(2,1,1)$&$\mathbb{F}(\ha_{-1}+\ha_1+\ha_0)\oplus\mathbb{F}\hat{q}$\\ \hline
$\bfournpodd{-1}^{\times}$&6&4&4&$(3,0,1)$&$\mathbb{F}s_i\oplus\mathbb{F}\hat{q}$ with $i=0$ or $1$\\ 
&&&&or $(2,1,1)$&\\ \hline
$\fournp{-1}^{\times}$&6&4&5&$(3,1,1)$&$\mathbb{F}\hat{q}$
\end{tabular}
}}
\begin{tablenotes}
\item
\end{tablenotes}
\caption{Quotients of $\fournp{-1}$}
\label{tabfialg2}
\end{center}
\end{threeparttable}

\begin{remark}
The algebras other than $\onea$, $\twoa$, $\threec{\eta}$ and $\threec{-1}^{\times}$ are not primitive.
The algebras $\honea$, $\bfournpmin{\eta}$, $\bfournpmin{-1}^{\times}$, $\bfournpodd{\eta}$ and $\bfournpodd{-1}^{\times}$ are not dihedral.
\end{remark}

Let $M=\fournp{\eta}$.

\begin{lemma}\sl
Assume that $R$ is a subspace of $M$ such that $a_1-a_{-1}\notin R$.

Then $R$ is an ideal of $M$ if and only if
$$R=\bigoplus_{\iota_0,\iota_1\in\{0,1\}}(\eisp{\ha_0+s_0}{\iota_0}\cap\eisp{\ha_1+s_1}{\iota_1}\cap R).$$
\end{lemma}
\begin{proof}
Let $\iota_0,\iota_1\in\{0,1\}$ and $x\in\eisp{\ha_0+s_0}{\iota_0}\cap\eisp{\ha_1+s_1}{\iota_1}$.
Then $x\in\eisp{\ha_{-1}+s_0}{\iota_0}\cap\eisp{\ha_{-1}+s_1}{\iota_1}$ by the invariance of $x$.
Since
\begin{eqnarray*}
\iota_0\iota_1 x&&=(\ha_0+s_0)((\ha_1+s_1)x)\\
&&=((\ha_0+s_0)(\ha_1+s_1))x\\
&&=(-\frac{1}{2}(\hat{q}+(\eta+1)(s_0+s_1))+\frac{\eta+1}{2}((\ha_0+s_0)+(\ha_1+s_1))+\frac{1-\eta}{2}\ha_{-1})x\\
&&=\frac{\eta+1}{2}(\iota_0+\iota_1-1)x+\frac{1-\eta}{2}\ha_{-1}x,
\end{eqnarray*}
$$\ha_{-1}x=\frac{1}{\eta-1}(-2\iota_0\iota_1+(\eta+1)(\iota_0+\iota_1+1))x$$
and thus
$$s_ix=\frac{1}{\eta-1}(2\iota_i\iota_j-2\iota_i-(\eta+1)(\iota_j+1))x$$
where $\{i,j\}=\{0,1\}$.
Thus $\mathbb{F}x$ is an ideal of $M$.
So $R$ is an ideal of $M$ if
$$R=\bigoplus_{\iota_0,\iota_1\in\{0,1\}}(\eisp{\ha_0+s_0}{\iota_0}\cap\eisp{\ha_1+s_1}{\iota_1}\cap R).$$

Assume that $R$ is an ideal of $M$ such that $a_1-a_{-1}\notin R$.
By considering the eigenspaces of $M/R$, we can easily show that $R$ is invariant under Miyamoto involutions.

Assume that $R\cap\eisp{\ha_1+s_1}{\eta}\neq0$.
Then, by the invariance of $R$, $\ha_1-\ha_{-1}\in R$.
So $R\cap\eisp{\ha_1+s_1}{\eta}=0$.
Recall that for $x\in R$ and $i\in\{0,1\}$,
$$\frac{1}{\eta(\eta-1)}(\ha_i+s_i)(\ha_i+s_i-1)x\in\eisp{\ha_i+s_i}{\eta}$$
and
$$x-\frac{1}{\eta(\eta-1)}(\ha_i+s_i)((\ha_i+s_i-1)x)\in\eisp{\ha_i+s_i}{0}\oplus\eisp{\ha_i+s_i}{1}.$$
Since $R\cap\eisp{\ha_i+s_i}{\eta}=0$ for $i\in\{0,1\}$, 
$$R\subset(\eisp{\ha_0+s_0}{0}\oplus\eisp{\ha_0+s_0}{1})\cap(\eisp{\ha_1+s_1}{0}\oplus\eisp{\ha_1+s_1}{1}).$$
Since $\eisp{\ha_i+s_i}{0}\eisp{\ha_i+s_i}{1}=\emptyset$ for $i\in\{0,1\}$, it is known that
$$(\ha_i+s_i)(xy)=x((\ha_i+s_i)y)$$
for all $i\in\{0,1\}$, $x\in\eisp{\ha_i+s_i}{0}\oplus\eisp{\ha_i+s_i}{1}$ and $y\in M$.
This condition is called Seress condition.
Thus, 
$$(\ha_0+s_0)((\ha_1+s_1)x)=((\ha_0+s_0)(\ha_1+s_1))x=(\ha_1+s_0)((\ha_1+s_1)x)$$
for all $x\in(\eisp{\ha_0+s_0}{0}\oplus\eisp{\ha_0+s_0}{1})\cap(\eisp{\ha_1+s_1}{0}\oplus\eisp{\ha_1+s_1}{1})$.
Thus, for $x\in R$,
$$(\ha_0+s_0-1)((\ha_1+s_1-1)x)\in\eisp{\ha_0+s_0}{1}\cap\eisp{\ha_1+s_1}{1}\cap R,$$
$$(\ha_0+s_0-1)((\ha_1+s_1)x)\in\eisp{\ha_0+s_0}{1}\cap\eisp{\ha_1+s_1}{0}\cap R,$$
$$(\ha_0+s_0)((\ha_1+s_1-1)x)\in\eisp{\ha_0+s_0}{0}\cap\eisp{\ha_1+s_1}{1}\cap R$$
and
$$(\ha_0+s_0)((\ha_1+s_1)x)\in\eisp{\ha_0+s_0}{0}\cap\eisp{\ha_1+s_1}{0}\cap R.$$
Furthermore,
\begin{eqnarray*}
x=&&(\ha_0+s_0-1)((\ha_1+s_1-1)x)+(\ha_0+s_0-1)((\ha_1+s_1)x)\\
&&+(\ha_0+s_0)((\ha_1+s_1-1)x)+(\ha_0+s_0)((\ha_1+s_1)x).
\end{eqnarray*}
Thus
$$R\subset\bigoplus_{\iota_0,\iota_1\in\{0,1\}}(\eisp{\ha_0+s_0}{\iota_0}\cap\eisp{\ha_1+s_1}{\iota_1}\cap R).$$

\end{proof}

\begin{proof}[Proof of Lemma \ref{lemquo}]
Let
$$r=\hat{q}-(\ha_{-1}+\ha_1+\ha_0+s_0).$$
Then 
$$\mathbb{F}s_0=\eisp{\ha_0+s_0}{1}\cap\eisp{\ha_1+s_1}{0},$$
$$\mathbb{F}s_1=\eisp{\ha_0+s_0}{0}\cap\eisp{\ha_1+s_1}{1}$$ 
and 
$$\mathbb{F}r=\eisp{\ha_0+s_0}{1}\cap\eisp{\ha_1+s_1}{1}.$$
Furthermore, 
$$\eisp{\ha_0+s_0}{0}\cap\eisp{\ha_1+s_1}{0}=0$$
if $\eta\neq-1$ and
$$\eisp{\ha_0+s_0}{0}\cap\eisp{\ha_1+s_1}{0}=\mathbb{F}\hat{q}$$
if $\eta=-1$.
So an ideal $R$ such that $a_1-a_{-1}\notin R$ is spanned by a subset of $\{s_0,s_1,r\}$ if $\eta\neq-1$ and spanned by a subset of $\{s_0,s_1,r,\hat{q}\}$ if $\eta=-1$.

Assume that $\ha_1-\ha_{-1}\in R$.
Since $R/(\ha_1-\ha_{-1})\cong\htwoa$, $M/R$ is isomorphic to a quotient of $\twoa$.
Then  by Proposition \ref{qtwoa}, a quotient of $\htwoa$ is isomorphic to $\htwoa$, $\honea$, $\twoa$ or $\onea$

Then the proposition holds and the classification of 2-generated $\mathcal{F}_\Phi$-axial algebras with $\Phi\neq\{0,1\}$ and $\eta\neq\frac{1}{2}$ is completed.
\end{proof}

\end{document}